\documentclass[11pt,a4paper, twoside]{amsart}
\usepackage{url}
\usepackage{amssymb}
\usepackage{color}
\usepackage{fancyhdr}
\usepackage{colonequals}
\usepackage{graphicx}
\usepackage{pinlabel}

\input{epsf.tex}

\renewcommand{\epsilon}{\varepsilon}

\renewcommand{\emptyset}{\varnothing}

\newtheorem*{namedtheorem}{\theoremname}
\newcommand{\theoremname}{testing}
\newenvironment{named}[1]{\renewcommand{\theoremname}{#1}\begin{namedtheorem}}{\end{namedtheorem}}

\newtheorem{theorem}{Theorem}[section]
\newtheorem{proposition}[theorem]{Proposition}
\newtheorem{corollary}[theorem]{Corollary}
\newtheorem{lemma}[theorem]{Lemma}

\newtheorem{problem}[theorem]{Problem}

\theoremstyle{definition}
\newtheorem{example}[theorem]{Example}

\newtheorem{definition}[theorem]{Definition}

\theoremstyle{remark}
\newtheorem{remark}[theorem]{Remark}

\newcommand{\Z}{\mathbb Z}

\newcommand{\R}{\mathbb R}

\newcommand{\CI}{\mathcal I}

\newcommand{\SCI}{\mathrm S\mathcal I}

\newcommand{\GL}{\operatorname{GL}}
\newcommand{\SL}{\operatorname{SL}}

\newcommand{\cohom}[3]{H^{{\raise1pt\hbox{$\scriptstyle#1$}}}(#2\>\!,#3)}
\newcommand{\tatecohom}[3]%
  {\widehat H^{{\raise1pt\hbox{$\scriptstyle#1$}}}(#2\>\!,#3)}

\newcommand{\Cohom}[3]%
  {H^{{\raise1pt\hbox{$\scriptstyle#1$}}}\big(#2\>\!,#3\big)}
\newcommand{\Tatecohom}[3]%
  {\widehat H^{{\raise1pt\hbox{$\scriptstyle#1$}}}\big(#2\>\!,#3\big)}

\newcommand{\homol}[3]{H_{{\lower1pt\hbox{$\scriptstyle#1$}}}(#2\>\!,#3)}
\newcommand{\homolog}[2]{H_{{\lower1pt\hbox{$\scriptstyle#1$}}}(#2)}

\renewcommand{\implies}{\Rightarrow}

\DeclareMathOperator{\Out}{Out}
\DeclareMathOperator{\Mod}{Mod}

\DeclareMathOperator{\Aut}{Aut}
\DeclareMathOperator{\SAut}{SAut}

\DeclareMathOperator{\Inn}{Inn}

\newcommand{\lk}{\operatorname{lk}}
\newcommand{\st}{\operatorname{st}}

\begin{document}

\title[First cohomology of automorphism groups of RAAGs]{On the first cohomology  of automorphism groups of graph groups}

\author{Javier Aramayona \& Conchita~Mart\'inez-P\'erez}

\begin{abstract}

We study the (virtual) indicability of the automorphism group $\Aut(A_\Gamma)$ of the right-angled Artin group $A_\Gamma$ associated to a simplicial graph $\Gamma$. 
First, we identify two conditions -- denoted (B1) and (B2) --  on $\Gamma$ which together imply that $H^1(G, \Z)=0$ for certain finite-index subgroups $G<\Aut(A_\Gamma)$.  
On the other hand we will show that (B2) is equivalent to the matrix group ${\mathcal H} = {\rm Im}(\Aut(A_\Gamma) \to \Aut(H_1(A_\Gamma))) <\GL(n,\Z)$ not being virtually indicable, and also to $\mathcal H$ having Kazhdan's property (T). As a consequence, $\Aut(A_\Gamma)$ virtually surjects onto $\Z$ whenever $\Gamma$ does not satisfy (B2). 
In addition, we give an extra property of $\Gamma$  ensuring that  $\Aut(A_\Gamma)$ and $\Out(A_\Gamma)$ virtually surject onto $\Z$. 
Finally, in the appendix we offer some remarks on the linearity problem for $\Aut(A_\Gamma)$. 

%We study the (virtual) indicability of the automorphism group $\Aut(A_\Gamma)$ of the right-angled Artin group $A_\Gamma$ associated to a simplicial graph $\Gamma$.
%More concretely, we identify a condition on $\Gamma$ which implies that $H^1(G, \Z)=0$ for certain finite-index subgroups $G<\Aut(A_\Gamma)$. On the other end of the spectrum, we exhibit two other conditions which
% guarantee that $\Aut(A_\Gamma)$, as well as a class of its finite-index subgroups, surject onto $\Z$. As a consequence, we obtain classes of graphs $\Gamma$ such that $\Aut(A_\Gamma)$ and $\Out(A_\Gamma)$ do not have Kazhdan's Property (T). 
%Finally, in the appendix we offer some remarks on the linearity problem for $\Aut(A_\Gamma)$. 

\end{abstract}

\maketitle

\section{Introduction}

Automorphism groups of right-angled Artin groups (or {\em graph groups}, or {\em partially commutative groups}) form an interesting class of groups, as they ``interpolate" between the two extremal cases of $\Aut(F)$, the automorphism group of a  non-abelian  free group, and the general linear group $\GL(n,\Z)$. 

In this paper we study the (non-)triviality of the first cohomology group of $\Aut(A_\Gamma)$ and certain classes of its finite-index subgroups; here, $A_\Gamma$ denotes the right-angled Artin group defined by the simplicial graph $\Gamma$. Recall that a discrete group $G$ is said to be {\em virtually indicable} if there exists a subgroup $G_0<G$ of finite index with non-trivial first cohomology group; equivalently, $G_0$ admits a surjection onto $\Z$. We say that a compactly generated group $G$ has {\em Kazhdan's property (T)} if every unitary representation of $G$ that has almost invariant vectors has an invariant unit vector. 
Groups with  {\em Kazhdan's property (T)} are not virtually indicable, however the converse is not true; for instance, $\Aut(F_3)$ has finite abelianization but does not enjoy property (T) \cite{McCool,GL,BV}.

%Before stating our result we need some definitions. Let $H$ be a compactly generated group which acts on a Hilbert space 
%$  \mathcal H$ by unitary transformations. We say that $H$ has {\em almost invariant vectors } if for every compact set 
%$C\subset G%$ and every $\epsilon >0$, there exists a unit vector $V\in \mathcal H$ with $||hV - V|| <\epsilon$ for all $h \in C$. %We say that a compactly generated group $G$ has {\em Kazhdan's property (T)} if every unitary representation of $G$ that has %almost invariant vectors has an invariant unit vector. 
% It is known that if a  discrete group has property (T) then each one of its finite-index subgroups
%has finite abelianization. 

As is often the case with properties of (automorphisms of) right-angled Artin groups, whether $H^1(\Aut(A_\Gamma), \Z)$ vanishes or not depends on the structure of the underlying graph $\Gamma$. Below, we will identify a number of conditions on $\Gamma$ ensuring that $\Aut(A_\Gamma)$ has (non-)trivial first cohomology. These conditions are phrased on the usual partial ordering of the vertex set $V(\Gamma)$ of $\Gamma$. Namely, given  vertices $v,w \in V(\Gamma)$, we  say that $ v\le w$ if $\lk(v) \subset \st(w)$; see section \ref{s:defs} for an expanded definition. We write $v \sim w$ to mean $v\le w$ and $w \le v$.

%%%%%%%%%%%%%%%%%%%%%%%%%%%%%%%%%%%%%%%%%%%%%%%%%%%
%%%%%%%%%%%%%%%%%%%%%%%%%%%%%%%%%%%%%%%%%%%%%%%%%%
%                                      			FINITE ABELIANIZATION

\subsection{Finite abelianization} We first consider a property of a graph which guarantees that the partial ordering $\le$ is ``sufficiently rich".  More concretely, we say that a simplicial graph $\Gamma$ has property (B) if the following two conditions hold:
\begin{enumerate}
\item[(B1)] For all $u,v \in V(\Gamma)$ that are not adjacent, we have $u \sim v$;
\item[(B2)] For all $v,w \in V(\Gamma)$ with $v\le w$, there exists $u\in V(\Gamma)$ such that  $u \ne v,w$ and $v\le u\le w$. 
\end{enumerate} 

We will prove that if $\Gamma$ has property (B) then  a natural  class of finite-index subgroups of $\Aut(A_\Gamma)$ have finite abelianization. Before we state our first result, 
recall that the Torelli group $\CI A_\Gamma$ is the kernel of the natural homomorphism $\Aut(A_\Gamma) \to \Aut(H_1(A_\Gamma))= \GL(n,\Z)$ where $n$ denotes the number of vertices of $\Gamma$. In particular, we may see  $\Aut(A_\Gamma)/\CI A_\Gamma$ as a subgroup of $\GL(n,\Z)$.

Our first result is: 

\begin{theorem}
Let $\Gamma$ be a simplicial graph with property (B). If $G < \Aut(A_\Gamma)$ is a finite-index subgroup containing $\CI A_\Gamma$, then $H^1(G, \Z) = 0$.
\label{thma}
\end{theorem}

 \begin{remark} Denote by $F_k$ the free group on $k$ letters. As we will see in Lemma \ref{propertyB} below, $\Gamma$ has property (B) if and only if   $A_\Gamma \cong F_{n_1} \times \ldots \times F_{n_k}\times\Z^a,$ where $n_1,\ldots,n_k> 2$ and $a\neq 2$. Observe, however, that if $w$ is a vertex corresponding to the abelian factor then for any other vertex $v$ there is a transvection $t_{vw}\in \Aut(A_\Gamma)$ mapping $v\mapsto vw$ and fixing the rest of the generators. In particular, $\Aut(A_\Gamma)$ does not keep the factors invariant in general; this serves to highlight the (well-known) fact that the automorphism group of such a direct product is a lot more complicated than the product of the automorphism groups of the factors. 
\label{rmk1}
\end{remark}

In view of the remark above, one sees that a number of particular cases of Theorem \ref{thma} were previously known. Indeed, if $\Aut(A_\Gamma) = \Aut(F_n)$ and $n\ge 3$, then our main result appears as Theorem 4.1 of  \cite{BV}. Moreover, Satoh \cite{satoh} has described the (finite)
abelianization of the {\em m-Torelli subgroup} $\CI A(m)<\Aut(F_n)$, also for $n \ge 3$; see section \ref{s:defs} for definitions. We stress, however, that deciding whether {\em every} finite-index subgroup of $\Aut(F_n)$ has finite abelianization -- or, more generally,  {\em Kazhdan's property (T)}, see below -- for $n\ge 4$ remains a challenging open problem. The answer is known to be negative for $n=3$ (see \cite{McCool,GL,BV}), and for $n=2$ since $\Aut(F_2)$ surjects onto a free group. 

On the other end of the spectrum, if $\Gamma$ is a complete graph on $n\ge 3$ vertices  -- in which case it also satisfies property (B) -- then every finite-index subgroup of $\Aut(A_\Gamma) = GL(n, \Z)$ has finite abelianization, for instance because $\GL(n,\Z)$ has {\em Kazhdan's property (T)}  \cite{Kahzdan}.

%%%%%%%%%%%%%%%%%%%%%%%%%%%%%%%%%%%%%%%%%%%%%%%%%%%%%%%%%%%%%
%%%%%%%%%%%%%%%%%%%%%%%%%%%%%%%%%%%%%%%%%%%%%%%%%%%%%%%%%%%%%

%							PROPERTY T

\subsection{Property (T)} In the opposite direction to Theorem \ref{thma}, we will show that the failure of property (B2) has strong consequences for the matrix group $\Aut(A_\Gamma) / \CI A_\Gamma< \GL(n,\Z)$, with $n$ the number of vertices of $\Gamma$, which will in particular imply that an explicit finite-index subgroup $\Aut(A_\Gamma)$ surjects onto $\Z$. 

Following the notation from \cite{Day} we denote by $\SAut^0(A_\Gamma)$ the finite-index subgroup of $\Aut(A_\Gamma)$  generated by {\em transvections} and {\em partial conjugations}; see section \ref{s:defs} for definitions. By Corollary 4.10 of \cite{Wade}, $\CI A_\Gamma < \SAut^0(A_\Gamma)$. We will prove:

\begin{theorem} Let $\Gamma$ be a simplicial graph. Then the following conditions are equivalent:

\begin{itemize}
\item[i)] $\Gamma$ has property (B2),

\item[ii)] $\Aut(A_\Gamma)/\CI A_\Gamma$  has Kazhdan's property (T),

\item[iii)] $\Aut(A_\Gamma)/\CI A_\Gamma$ is not virtually indicable,

\item[iv)] $H^1(\SAut^0(A_\Gamma)/\CI A_\Gamma,\Z) = 0.$
\end{itemize}
\label{prop1}
\end{theorem}

Note that, in particular, Theorem \ref{prop1} gives a characterization of property (T) for matrix subgroups $\Aut(A_\Gamma)/\CI A_\Gamma < \GL(n,\Z)$ in terms of their virtual indicability. Observe also that, since property (T) is inherited by quotients, Theorem \ref{prop1} immediately yields: 

\begin{corollary}
Let $\Gamma$ be a simplicial graph  which does not satisfy property (B2). Then $\SAut^0(A_\Gamma)$ surjects onto $\Z$; consequently, $\SAut^0(A_\Gamma)$ does not have Kazhdan's property (T).
\label{thm-indic}
\end{corollary}

In the light of this corollary, a natural problem is: 

\begin{problem}
Suppose that the graph $\Gamma$ satisfies (B2) but not (B1). Does $\Aut(A_\Gamma)$ have Kazhdan's property (T)?   Is $\Aut(A_\Gamma)$ virtually indicable? 
\label{qn}
\end{problem}

We note that an affirmative answer to the second question in Problem \ref{qn}  implies a negative answer to the first.

%%%%%%%%%%%%%%%%%%%%%%%%%%%%%%%%%%%%%%%%%%%%%
%%%%%%%%%%%%%%%%%%%%%%%%%%%%%%%%%%%%%%%%%%%%

% 						VIRTUAL INDICABILITY	

%\marginpar{\tiny{Here, I changue $\Aut^0$ to $\Aut^1$ and also modify the statement of  Theorem \ref{thm-indic2}}}

\subsection{Virtual indicability}
As it turns out, the answer to the second  question in Problem \ref{qn} is affirmative for certain classes of graphs. This will come as a consequence of our next result, which asserts that the existence of a vertex $w\in V(\Gamma)$ that is  {\em minimal} with respect to the partial order $\le$, and such that $\Gamma - \st(w)$ is disconnected, also implies the virtual indicability of $\Aut(A_\Gamma)$. More concretely, let $\Aut^1(A_\Gamma)$ be the finite-index subgroup of $\Aut(A_\Gamma)$  generated by {\em transvections}, {\em partial conjugations}, and the so called {\em thin inversions}; see again section \ref{s:defs} for definitions. We will show:

\begin{theorem}
Let $\Gamma$ be a simplicial graph. Suppose there exists a vertex $w\in V(\Gamma)$ such that
there is no $v \in V(\Gamma)$ with $v\le w$. Then $\Aut^1(A_\Gamma)$  surjects onto $\Z$. If moreover
$\Gamma - \st(w)$ is not connected, then also $\Out^1(A_\Gamma)$ surjects onto $\Z$.

\label{thm-indic2}
\end{theorem}

Here, $\Out^1(A_\Gamma)$ is the quotient of $\Aut^1(A_\Gamma)$ by the normal subgroup of inner automorphisms which has finite index in the  group $\Out(A_\Gamma)$ of {\em outer} automorphisms of $A_\Gamma$. In a previous version of this paper (which includes the published version), a result implying Theorem \ref{thm-indic2} was stated with the group $\Aut^0(A_\Gamma)$ generated by    transvections, partial conjugations, and {\em all}  inversions instead of $\Aut^1(A_\Gamma)$. Unfortunately, there was a gap in the proof; nevertheless, the same argument works after passing down to the group $\Aut^1(A_\Gamma)$; see Section 5. 

\iffalse{\begin{remark}
In fact, the proof of Theorem \ref{thm-indic2} will also yield the virtual indicability of the  group $\Out(A_\Gamma)$ of {\em outer} automorphisms of $A_\Gamma$; see Corollary \ref{cor:out} below. 
\label{rmkout}
\end{remark}\fi

As an immediate consequence, we obtain: 

\begin{corollary} 
Let $\Gamma$ be a simplicial graph as in  Theorem \ref{thm-indic2}. Then $\Aut(A_\Gamma)$ and  $\Out(A_\Gamma)$ do not have Kazhdan's property (T). 
\label{corT2}
\end{corollary}

As mentioned above, Theorem \ref{thm-indic2} provides examples of graphs for which Question \ref{qn} has a positive answer:

\begin{example}
Consider the graph $\Gamma$ which is the disjoint union of a single vertex $\{w\}$, and a complete graph $\Gamma_1$ of at least  three vertices. Then $\Gamma$ satisfies (B2), but not (B1) since $w$ is not equivalent to any other vertex. As $w$ is a minimal element for $\le$, it follows from Theorem \ref{thm-indic2} that $\Aut^1(A_\Gamma)$ surjects onto $\Z$.
If we add a second complete graph  $\Gamma_2$ having at least three vertices too, then Theorem \ref{thm-indic2} implies that also $\Out^1(A_\Gamma)$ surjects onto $\Z$.
\end{example}

Finally we remark that, in the particular case when $\Gamma$ is a tree, it is possible to give a explicit characterization of the hypotheses of Theorem \ref{thm-indic2}; see Proposition \ref{tree-charac} below. 

\medskip

\subsubsection{Further applications} Theorems \ref{prop1} and \ref{thm-indic2} have a number of other immediate consequences which highlight that the behaviour of $\Aut(A_\Gamma)$ for an arbitrary $\Gamma$ can be quite different to that of the two extremal cases of $\Aut(F_n)$ and $\GL(n,\Z)$, as we now discuss. For the sake of concreteness, $\Gamma$ denotes a finite simplicial graph that does not satisfy property (B2); as we will see in the proof of Theorem \ref{prop1}, this yields the existence of a transvection that maps non-trivially under the homomorphism  $\SAut^0(A_\Gamma) \to \Z.$ 

\smallskip

\noindent{\em A. Actions of $\Aut(A_\Gamma)$ on non-positively curved spaces.} Bridson has proved that, whenever $n\ge 4$, if any finite index subgroup of $G<\Aut(F_n)$ acts on a complete non-positively curved  space (i.e.  a complete ${\rm CAT}(0)$ space) by (semisimple) isometries then every power of a transvection that lies in $G$ fixes a point; see Theorem 1.1 of \cite{Bridson}.
%has zero translation length; see Proposition 8.2 of \cite{Bridson}. Recall that the translation length of an isometry $\gamma$ of a metric space $(X,d)$ is defined as $$\inf\{d(x,\gamma\cdot x) \mid x\in X\}.$$
However, in the light of Theorem \ref{prop1} the analogous statement does not hold for arbitrary  $\Aut(A_\Gamma)$. Indeed, using the existence of a transvection $t\in \SAut^0(A_\Gamma)$ that maps nontrivially to $\Z$, we can construct (say) a non-trivial homomorphism $$\SAut^0(A_\Gamma) \to {\rm Isom}^+(\mathbb H^2)=\SL(2,\R)$$ such that the image of $t$ is a hyperbolic isometry. 

\smallskip

\noindent{\em B. Homomorphisms from $\Aut(A_\Gamma)$ into mapping class groups.} Recall that the mapping class group $\Mod(S)$ of a topological surface $S$ is the group of homeomorphisms of $S$, modulo isotopy. As a consequence of the above result of Bridson, one deduces that any homomorphism from (a finite index subgroup of) $\Aut( F_n) $ to $\Mod(S)$ must send (powers of) transvections to roots of Dehn multitwists; see Corollary 1.2 of \cite{Bridson}. The analogous statement for $\SL(n,\Z)$ is due to Farb-Masur \cite{FM}.  

Again using the existence of a transvection $t\in \SAut^0(A_\Gamma)$ that maps nontrivially to $\Z$, we deduce that there exist homomorphisms $\SAut^0(A_\Gamma) \to \Mod(S)$ for which the image of $t$ is a pseudo-Anosov (in particular, not a root of a multitwist). 

\smallskip

\noindent{\em C. Representations of $\Aut(A_\Gamma)$ into $\SL(m,\R)$.} Again in the same paper (see Corollary 8.3 of \cite{Bridson}), Bridson deduces that the image of a (power of a) transvection under a representation from (a finite index subgroup of) $\Aut(F_n)$ into $\SL(m, \R)$, where $n\ge 6$ and $m$ is arbitrary, is {\em unipotent}: all its eigenvalues are roots of unity. For $\SL(n,Z)$, the analogous statement follows from Margulis' {\em Superrigidity Theorem} \cite{Margulis}. However, using the same argument as in (A), we see that this is not true for arbitrary $\Aut(A_\Gamma)$.

\subsection{Linearity problem for $\Aut(A_\Gamma)$}
In a different direction, during our work we noticed that if the graph $\Gamma$ satisfies a certain (drastic) weakening of property (B1) above then $\Aut(A_\Gamma)$ is not linear; this refines a question of Charney (see Problem 14 of \cite{Charney}).  More concretely, we say that a simplicial graph $\Gamma$ satisfies property (NL) if there exist pairwise non-adjacent vertices $v_1, v_2, v_3 \in V(\Gamma)$ such that $v_3 \le v_i$, for $i=1,2$.  Using an argument of Formanek-Procesi \cite{FP}, we will observe:

\begin{proposition}
Let  $\Gamma$ be a simplicial graph with  property (NL). Then $\Aut(A_\Gamma)$  is not linear. 
\label{nonlinear}
\end{proposition}

%We will also  say some words about the probability that a given graph satisfies property (NL).

The plan of the paper is as follows. In section \ref{s:defs} we will give all the necessary definitions and results that will be used throughout the paper. 
Sections 2, 3, and 4 are devoted to the proofs of Theorem \ref{thma}, \ref{prop1}, and \ref{thm-indic2}, respectively. Finally, in the appendix we will discuss the linearity problem for automorphism groups of right-angled Artin groups.

\medskip

\noindent {\bf Acknowledgements.} The first author was supported by a 2014 Campus Iberus grant, and thanks the Universidad de Zaragoza for its hospitality. The second author was partially supported by  Gobierno de Arag\'on, European Regional 
Development Funds and MTM2010-19938-C03-03.   The authors would like to thank Matt Day, Andrei Jaikin, Juan Souto and Ric Wade for conversations. 
We are grateful to Dawid Kielak for suggesting us to include the application to Property (T), and to Neil Fullarton for pointing out a mistake in an earlier version of this paper. Finally, we thank the referee for useful comments and suggestions. 

%%%%%%%%%%%%%%%%%%%%%%%%%%%%%%%%%%%%%%%%%%%%%%%%

%%%

%%%%%%%%%%%%%%%%%%%%%%%%%%%%%%%%%%%%%%%%%%%%%%%%

\section{Definitions}
\label{s:defs}

\subsection{Graphs} Let $\Gamma$ be a  simplicial graph, and  denote its vertex set by $V(\Gamma)$. Given $v \in V(\Gamma)$, the {\em link} $\lk(v)$ of $v$ is the full subgraph spanned by those vertices of $\Gamma$ that are adjacent to $v$. The {\em star} $\st(v)$ of $v$ is defined as the subgraph spanned by the vertices in $\lk(v) \cup \{v\}$. 
As mentioned in the introduction, there is a natural partial ordering on $V(\Gamma)$ given by $$v\le w \iff \lk(v) \subset \st(w),$$
for any two vertices $v,w\in V(\Gamma)$. We will write $v\sim w$ to mean $v\le w$ and $w\le v$ and $[v]$ for the equivalence class of $v$ with respect to the relation $\sim$.

\subsection{Right-angled Artin groups} The {\em right-angled Artin group} defined by $\Gamma$ is the group given by the presentation
$$A_\Gamma= \langle v\in V(\Gamma) \mid [u,v]=1 \iff  u {\text{ and }} v {\text{ are connected by an edge}} \rangle.$$
Observe that if $\Gamma$ consists of $n$ isolated vertices, then $A_\Gamma = F_n$, the free group on $n$ letters. On the other end of the spectrum, if $\Gamma$ is a complete graph on $n$ vertices then $A_\Gamma= \Z^n$. 

\subsection{Automorphisms of right-angled Artin groups} Let $A_\Gamma$ be the right-angled Artin group defined by the finite simplicial graph $\Gamma$, and $\Aut(A_\Gamma)$ its automorphism group.   Laurence \cite{Laurence} and Servatius \cite{Servatius} proved that $\Aut(A_\Gamma)$ is generated by the following types of automorphisms: 

\medskip

\begin{itemize}

\item  {\em Graphic automorphisms:} Every isomorphism $\Gamma \to \Gamma$ gives rise to an automorphism of $A_\Gamma$, called a {\em graphic automorphism}. 

\medskip

\item {\em Inversions:} Given $v\in V(\Gamma)$, the {\em inversion} $\iota_v$ is the automorphism of $A_\Gamma$ defined by $$\Bigg\{\begin{aligned}
\iota_v(v)&=v^{-1}\\
\iota_v(z)&=z,\, z\neq v.\\
\end{aligned}$$

\medskip

\item {\em Transvections:} Let $v,w \in V(\Gamma)$ with $v\le w$. The {\em transvection} $t_{vw}$ is the automorphism of $A_\Gamma$ defined by $$\Bigg\{\begin{aligned}
t(v)&=vw\\
t(z)&=z,\, z\neq v.\\
\end{aligned}$$

\medskip

\item {\em Partial conjugations:} Let $v\in V(\Gamma)$ and $Y$ a connected component of $\Gamma - \st(v)$. The {\em partial conjugation} $c_{v,Y}$ is the automorphism of $A_\Gamma$  defined by $$\Bigg\{\begin{aligned}
c_{v,Y}(w)&=v^{-1}wv, w\in Y\\
c_{v,Y}(z)&=z,\, z\notin Y.\\
\end{aligned}$$

\end{itemize}

A {\sl thin} vertex $v\in V(\Gamma)$ is a vertex such that its equivalence class $[v]$ has only one element. An inversion $\iota$ is {\sl thin} if it fixes every thin vertex.

\subsection{A finite presentation of $\Aut(A_\Gamma)$} A central ingredient of the proof of Theorem \ref{thm-indic2} will be the finite presentation of $\Aut(A_\Gamma)$ computed by Day \cite{Day},  which we now describe.

Let $\Gamma$ be a simplicial graph. In order to relax notation, we will blur the difference between vertices of $\Gamma$ and generators of the right-angled Artin group $A_\Gamma$. Write $L = V(\Gamma) \cup V(\Gamma)^{-1} \subset A_\Gamma$. 

A {\em type (1) Whitehead automorphism} is an element $\alpha \in {\rm Sym}(L) \subset \Aut(A_\Gamma)$, where ${\rm Sym}(L)$ denotes the group of permutations of $L$. A {\em type (2) Whitehead automorphism} is specified by a subset $A\subset L$ and an element $v\in L$ with $v \in A$ but $v^{-1} \notin A$. Given these, we set $(A,v)(v) = v$ and, for $w \ne v$,

$$(A,v)(w)= \left \{ \begin{aligned}
w, & \text{ if } w \notin A \text{ and } w^{-1} \notin A \\
wv, & \text{ if } w \in A \text{ and } w^{-1} \notin A \\
v^{-1}w, & \text{ if } w \notin A \text{ and } w^{-1} \in A \\
v^{-1}wv, & \text{ if } w \in A \text{ and } w^{-1} \in A \\
\end{aligned}
\right.$$

Observe that the Laurence-Servatius generators described in the previous subsection are Whitehead automorphisms.
Indeed, if $v\in V(\Gamma)$ and $Y$ is a connected component of $\Gamma-\st(v)$, then the partial conjugation that conjugates all the elements in $Y$ by $v$ is
$$c_{v,Y} = (Y\cup Y^{-1}\cup\{v\},v).$$ Similarly, if 
 $v	\le w$ the transvection that maps $v\mapsto vw$ is
$$t_{vw} = (\{v,w\},w).$$

In \cite{Day}, Day proved:

\begin{theorem}[\cite{Day}]
$\Aut(A_\Gamma)$ is the group generated by the set of all Whitehead automorphisms, subject to the following relations: 

\begin{enumerate}

\item[(R1)] $(A,v)^{-1} = (A - v \cup v^{-1}, v^{-1})$,

\item[(R2)] $(A,v)(B,v) = (A \cup B, v)$ whenever $A \cap B = \{v\}$,

\item[(R3)] $(B, w)(A,v)(B,w)^{-1}  = (A,v)$, whenever $\{v,v^{-1}\} \cap B = \emptyset$, $\{w, w^{-1} \} \cap A = \emptyset$, and at least one of $A \cap B = \emptyset$ or $ w \in \lk(v)$ holds,

\item[(R4)] $(B, w)(A,v)(B,w)^{-1} =  (A,v)(B - w \cup v, v)$, whenever $\{v,v^{-1}\} \cap B = \emptyset$, $w \notin A$, $w^{-1} \in A$,  and at least one of $A \cap B = \emptyset$ or $ w \in \lk(v)$ holds,

\item[(R5)] $(A - v \cup v^{-1}, w)(A, v) = (A - w \cup w^{-1}, v) \sigma_{v,w}$, where $w \in A$,
$w^{-1} \notin A$, $w \ne v$ but $w \sim v$, and where $\sigma_{v,w}$ is the type (1) automorphism such that $\sigma_{v,w}(v) = w^{-1}$, $\sigma_{v,w}(w) = v$, fixing the rest of generators.

\item[(R6)] $\sigma (A,v) \sigma^{-1}= (\sigma(A), \sigma(v))$, for every $\sigma$ of type (1).
\item[(R7)] All the relations among type (1) Whitehead automorphisms.

\item[(R9)] $(A,v)(L - w^{-1}, w) (A,v)^{-1} = (L - w^{-1}, w)$, whenever $\{w,w^{-1}\} \cap A = \emptyset$, and

\item[(R10)]$(A,v)(L - w^{-1}, w) (A,v)^{-1}= (L - v^{-1}, v)( L - w^{-1}, w)$, whenever $w \in A$ and $w^{-1} \notin A$. 
\end{enumerate}
\label{thm-day}
\end{theorem}

\begin{remark}
In Day's list of relations \cite{Day} there is an extra type of relator, which Day calls (R8); however, as he mentions, this relation is redundant and therefore we omit it from the list above. 
\end{remark}

\subsection{Torelli group} Observe that $H_1(A_\Gamma, \Z) = \Z^n$, where $n$ is the number of vertices of $\Gamma$. Therefore we have a natural homomorphism 
$$\Aut(A_\Gamma) \to \Aut(H_1(A_\Gamma, \Z)) = \GL(n, \Z).$$ The kernel of this homomorphism, which we will denote by $\CI A_\Gamma$, is called the {\em Torelli subgroup} of $\Aut(A_\Gamma)$. Observe that every partial conjugation is an element of $\CI A_\Gamma$. We now describe another type of element of $\CI A_\Gamma$ that will be needed in the sequel. Let $u,v,w$ be vertices with  $\lk(v) \subset \st(u) \cap \st(w)$. Consider the automorphism $\tau_{u,v,w}$ of $A_\Gamma$ given by $$\Bigg\{\begin{aligned}
\tau_{u,v,w}(v)&=v[u,w]\\
\tau_{u,v,w}(z)&=z,\, z\neq v.\\
\end{aligned}$$
Every automorphism of the above form will be referred to as a $\tau$-map. Day proved that partial conjugations and $\tau$-maps suffice to generate $\CI A_\Gamma$; see Theorem B of \cite {Day2}, or Theorem 4.7 of \cite{Wade} for an alternate proof.

\begin{theorem}[\cite{Day2}]
 The Torelli group $\CI A_\Gamma$ is finitely generated by the set of partial conjugations and $\tau$-maps.
\label{torelligen}
\end{theorem}

%%%%%%%%%%%%%%%%%%%%%%%%%%%%%%%%

%%%%%%%%%%%%%%%%%%%%%%%%%%%%%%%%%%%%%%%%%%%%%%%%%%%%%%%%%%%%%%%%%%%%%%%%%%

\section{Proof of Theorem \ref{thma}}
\label{s:thma}

In this section we will give a proof of Theorem \ref{thma}. Recall from the introduction that a simplicial graph $\Gamma$ has property (B) if the following two conditions hold:
\begin{enumerate}
\item[(B1)] For all $u,v \in V(\Gamma)$ with $u \notin \lk(v)$, we have $u \sim v$;
\item[(B2)] For all $v,w \in V(\Gamma)$ with $v\le w$, there exists $u\in V(\Gamma)$ such that $u \ne v,w$ and  $v\le u\le w$. 
\end{enumerate} 

As mentioned in the introduction, the fact that the graph $\Gamma$ has property (B) turns out to be the graph-theoretic interpretation of the fact that $A_\Gamma$ splits as a direct product of free groups of rank different from 2. More concretely, we have:

\begin{lemma} Let $\Gamma$ be a finite simplicial graph, and $A_\Gamma$ the associated right-angled Artin group. 
 The following two statements are equivalent: 
 \begin{enumerate}
 \item  $\Gamma$ satisfies property (B). 
 \item   $A_\Gamma \cong F_{n_1} \times \ldots \times F_{n_k}\times\Z^a,$ where $n_1,\ldots,n_k> 2$ and $a\neq 2$. 
 \end{enumerate}
 If these statements hold, the numbers $n_1,\ldots, n_k$, and $a$ are the sizes of the equivalence classes $[v]$.
 \label{propertyB}
 \end{lemma}
  \begin{proof}
 The implication (ii) $\implies$ (i) follows from the {\em Isomorphism Theorem} of Droms \cite{Droms} and Laurence \cite{Laurence}: two right-angled Artin groups are isomorphic if and only if the defining graphs are isomorphic. 
 
 Suppose now that $\Gamma$ satisfies property (B).  Given a vertex $v\in V(\Gamma)$, denote by $\Gamma_{[v]}$  the full subgraph of $\Gamma$ spanned by the elements of $[v]$. 

 We claim that, for every $v\in V(\Gamma)$, the graph $\Gamma_{[v]}$ is either complete or totally disconnected. To see this, suppose that  $\Gamma_{[v]}$ has two vertices $w_1,w_2$ with $w_2\in \lk(w_1)$, and note that  $w_1 \in\lk(w_2)$ also. Consider a third vertex $u \ne w_1,w_2$ in $\Gamma_{[v]}$ . Since $u \sim w_1$ and $w_2 \in \lk(w_1)$ then $w_2 \in \st(u)$. Similarly, $w_1 \in \st(u)$, and hence the claim follows. 
 
 Second, observe that property (B1) implies that if $[v] \ne [w]$ then every vertex of  $\Gamma_{[v]}$ is adjacent to every vertex of $\Gamma_{[w]}$; indeed, if there exist $v' \in \Gamma_{[v]}$ and $w'\in \Gamma_{[w]}$ with $v' \notin \lk(w')$ then we would have $v'\sim w'$, by (B1), which contradicts the assumption $[v] \ne [w]$.  In particular, this also implies that there is at most one equivalence class $[v]$ for which $\Gamma_{[v]}$ is a complete graph. 
 
Now, the first claim above implies that $A_{\Gamma_{[v]}}$ is either a free abelian group (in the case when $\Gamma_{[v]}$ is a complete graph) 
 or a non-abelian free group (in the case when $\Gamma_{[v]}$  is totally disconnected). In turn, the second claim gives that if   $[v] \ne [w]$ then every element of  $A_{\Gamma_{[v]}}$ commutes with every element of $A_{\Gamma_{[w]}}$.  In other words, we have deduced that 
 \begin{equation}
A_\Gamma \cong F_{n_1} \times \ldots \times F_{n_k}\times \Z^a,
\label{eqdirect}
\end{equation}
where $n_1, \ldots, n_k$ are the cardinalities of equivalence classes whose elements span a totally disconnected graph, while $a$ is the cardinality of the equivalence class whose elements span a complete graph. 

Finally, observe that $n_i \ne 2$ for all $i$ in virtue of (B2) above. Indeed, if $F_{n_1} = \langle v, w \rangle$, we claim the vertex $u$ given in (B2) cannot exist. Otherwise, $u$ would belong to a different  factor of $A_\Gamma$ in the direct product decomposition given in equation (\ref{eqdirect}) above, which implies that $v \in \lk(u)$. However, $v\notin \st(w)$, which contradicts that $ u \le w$. 
Similarly,  we claim that $a\ne2$ also. To see this, suppose that $a=2$ and denote the abelian factor by $\langle v\rangle\times\langle w\rangle$. Then the vertex $u$ given by (B2) cannot exist, for if $v\leq u$ then $\lk(v)\subseteq\st(u)$, which implies  $u=w$ as $\lk(v)=\Gamma-v$.  This finishes the proof of the lemma.
 \end{proof}

 The key ingredient in the proof of Theorem \ref{thma} will be to understand the abelianization of a certain finite-index subgroup of the {\em $m$-Torelli subgroup} $\CI A_\Gamma (m)$ of $\Aut(A_\Gamma)$, which is defined as the kernel of the homomorphism 
$$\Aut(A_\Gamma) \to \GL(n, \Z) \to \GL(n, \Z_m),$$ where the second arrow is given by reducing matrix entries modulo $m$ and    $n$ is the number of vertices of $\Gamma$.
 
 We denote 
$$\SCI A_\Gamma(m):=\SAut^0(A_\Gamma)\cap\CI A_\Gamma(m).$$
Observe that this subgroup is the kernel of the homomorphism 
$$\SAut^0(A_\Gamma) \to \SL(n, \Z) \to \SL(n, \Z_m)$$
and that $\SCI A_\Gamma(m)$ has finite index in  $\Aut(A_\Gamma)$ for every $m\ge 0$. 
Now, for $n\geq 3$ the kernel of the homomorphism $\SL(n, \Z) \to \SL(n, \Z_m)$ is normally generated by $m$-powers of transvections in $\SL(n, \Z)$ \cite{BMS}; recall that a transvection in $\SL(n,\Z)$ is a matrix of the form $T_{ij}:=I_n + E_{ij}$ for $i\ne j$, where $I_n$ is the identity matrix and $E_{ij}$ is the matrix that has a 1 in the $(i,j)$ position and zeroes elsewhere. 

We need to analyze the image $\mathcal H$ of $\SAut^0(A_\Gamma)$ under the  natural homomorphism $\Aut(A_\Gamma) \to \GL(n,\Z)$, as done in \cite{Wade}. Consider the partition of $V(\Gamma)$ given by the equivalence relation $\sim$, and order the classes
%Group the vertices of $\Gamma$ in their equivalence classes $[v]$ respect to the  relation $\sim$. Order the cosets 
$\{[v_1],\ldots,[v_k]\}$ in ascending order, i.e. so that $v_j\leq v_i$ implies $j\leq i$. Then  there is a transvection $t_{vw}\in\SAut^0(A_\Gamma)$ if and only if $v\in[v_j]$ and $w\in[v_i]$ with $v_j\leq v_i$. As $\mathcal{H}$ is generated by the images of the transvections in $\SAut^0(A_\Gamma)$,  we see that any matrix in $\mathcal{H}$ is block lower triangular, with the sizes of the diagonal blocks corresponding to the cardinality of each of the clases. The transvections in $\mathcal{H}$ are precisely those of the form $T_{r_ic_j}=I_n + E_{r_ic_j}$ for some $(i,j)$ such that $v_j\leq v_i$ and with $r_i$ corresponding to the block $[v_i]$, $c_j$ corresponding to the block $[v_j]$. Moreover, for any matrix $M\in\mathcal{H}$ and  any $1\leq i,j\leq k$, if the $(i,j)$-block of $M$ is not zero then we must have $v_j\leq v_i$. The fact that the product of two matrices of this form is again a matrix of this form follows from the transitivity of the partial order $\leq$ on $V(\Gamma)$.

Using the previous facts we are going to prove the following technical lemma, which will be used in the proof of Theorem \ref{thma}:

\begin{lemma} Assume that there is no class $[v]$ with respect to $\sim$ that has exactly 2 elements. Then for all $m\ge 0$, $\SCI A_\Gamma(m)$ is normally generated by $\CI A_\Gamma$ and all $m$-th powers of transvections in $\Aut(A_\Gamma)$. 
\label{m-torelligen}
\end{lemma}
\begin{proof} Let $\mathcal{H}(m)$ be the kernel of the map $\mathcal{H}\to\SL(n,\Z_m)$ and $T$ be the group normally generated by the $m$-powers of the transvections in $\mathcal{H}$. We claim that $\mathcal{H}(m)=T$, which in turn implies the desired result.

 We obviously have $T\subseteq \mathcal{H}(m)$, so we only need to prove the reverse inclusion. Let $M$ be a matrix in $\mathcal{H}(m)$; by the above discussion $M$ is a block lower triangular matrix with diagonal blocks corresponding to the classes  $\{[v_1],\ldots,[v_k]\}$. Observe also that if there was some class with only one element then the corresponding block diagonal entry of $M$ would be  a single 1. In addition, note that all the matrices in the diagonal blocks of $M$ are a product of transvections in $\GL(n_i, \Z)$ where $n_i$ is the number of elements in the class $[v_i]$; in particular such diagonal blocks are elements of $\SL(n_i,\Z)$. 
Let $M_1$ be the matrix in the first diagonal block of $M$. As mentioned above, if $[v_1]$ has only one element, then $M_1$ is a 1; otherwise, the hypothesis implies that the size $n_1$ of $M_1$ is at least 3. Observe that $M_1$ lies in the kernel of the map from $\GL(n_1,\Z)$ to $\GL(n_1,\Z_m)$, and therefore $M_1$ is a  product of conjugates (in $\SL(n_1,\Z)$) of $m$-powers of transvections in $\SL(n_1,\Z)$.

Embed $\SL(n_1,\Z)$ into $\SL(n,\Z)$ via the first diagonal block. Via this embedding, the previous expression of $M_1$ as  a  product of conjugates of $m$-powers of transvections in $\SL(n_1,\Z)$ yields a matrix $N_1\in T$ such that  $MN_1$ has the first diagonal block equal to $I_{n_1}$. Now, we may repeat the argument with the rest of the diagonal blocks and find matrices $N_2,\ldots,N_k\in T$ such that  $Q=N_1\ldots N_k \in T$ (and in particular lies in $\mathcal{H}(m)$) and $MQ$ has every diagonal block equal to the identity matrix. 
To finish the proof, we claim that there is some matrix $P\in T$ such that $PMQ$ is the identity. This $P$ should be the result of left multiplying $MQ$   by elementary matrices in $T$ so that the associated row operation kills the non-diagonal entries.    To see that this is possible, note that  all the non-diagonal entries of $MQ$ are multiples of $m$ and that if there is some non-zero entry in the subblock $(i,j)$ then by the observations above over $\mathcal{H}$ we must have $v_j\leq v_i$. This in turn implies that any transvection $T^m_{r_ic_j}=I_n+mE_{r_ic_j}$ lies in $T$ and as these are precisely the kind of transvections that we need we get the result.
\end{proof}

As indicated above, the proof of Theorem \ref{thma} boils down to understanding the abelianization of $\SCI A_\Gamma(m)$, which we describe in the next proposition. Given a group $G$, we denote by $G'$ its commutator subgroup $ [G,G]$. Also, we denote by $G^{ab}$ the abelianization of $G$, i.e. $G^{ab}= G/G'$. We will show: 

\begin{proposition}
Let $\Gamma$ be a graph that satisfies property (B). For every $m\ge 0$, the abelianization  of $\SCI A_\Gamma(m)$ is a finite $m$-group. In other words, $\SCI A_\Gamma(m)^{ab}$ is finite and the order of every element divides $m$.
\label{m-torelliab}
\end{proposition}

\begin{remark}
As mentioned in the introduction, Proposition \ref{m-torelliab} is implied by the work of Satoh \cite{satoh} when $A_\Gamma$ is a free group.
\end{remark}

The proof of Proposition \ref{m-torelliab} is broken down into a series of lemmas. Throughout the remainder of the paper, given $\alpha, \beta \in \Aut(A_\Gamma)$, we will write $\alpha^\beta$ to mean the conjugate of $\alpha$ by $\beta$, i.e.  $\alpha^\beta:= \beta^{-1} \alpha \beta$. We begin with an easy consequence of (B1).

\begin{lemma} Let $\Gamma$ be a  graph that satisfies (B1). For any $v\in \Gamma$ such that $\st(v)\neq\Gamma$, we have $\Gamma-\st(v)$ is totally disconnected.
In particular, there is a partial conjugation in $\Aut(A_\Gamma)$ of the form $c_{v,\{u\}}$, with $u\not\in\st(v)$.
\end{lemma} 
\begin{proof} Let $u\in\Gamma-\st(v)$.  Then (B1) implies $u\sim v$, which in turn gives $\lk(u)\subseteq\st(v)$. Thus we obtain that $u$ is an isolated vertex in $\Gamma - \st(v)$.  
\end{proof}

In order to relax notation, we will write $c_{v,u}:=c_{v,\{u\}}$.  The following result is an analog in our context of the {\em crossed lantern relation} for the mapping class group, see \cite{Putman}.

\begin{lemma}
Let $\Gamma$ be a graph with property (B1). Let  $v,w\in V(\Gamma)$ with $v \notin \lk(w)$ and consider  $c_1:= c_{w,v}$. Then there exist $t,c_2\in \Aut(A_\Gamma)$ such that $c_1^m = [t^m, c_2^{-1}]$ for all $m\ge 0$. In particular, $c_1^m \in \SCI A'_\Gamma(m)$ for all $m\ge 0$. 
\label{crossed}
\end{lemma}

\begin{proof}
As $v \notin \lk(w)$, property (B1) implies that $v\sim w$. In particular, the transvection $t:=t_{vw^{-1}}$ 
is well-defined. Next, let $c_2:= c_{v,w}$.  We have   
$$tc_1(v)=t(w^{-1}vw)=w^{-1}t(v)w=w^{-1}v,$$
$$c_1t(v)=c_1(vw^{-1})=c_1(v)w^{-1}=w^{-1}v$$
thus $c_1=t^{-1}c_1t$.
On the other hand,
$$tc_1c_2(v)=tc_1(v)=w^{-1}v,$$
$$c_2t(v)=c_2(vw^{-1})=vv^{-1}w^{-1}v=w^{-1}v$$
and
$$tc_1c_2(w)=tc_1(v^{-1}wv)=t(w^{-1}v^{-1}www^{-1}vw)=v^{-1}wv=c_2(w)=c_2t(w),$$ so
 $c_1c_2=t^{-1}c_2t$.  Therefore $c_1=c_2^tc_2^{-1}$ and
$$\begin{aligned}c_1^m=c_1^{t^{m-1}}c_1^{t^{m-2}}\ldots c_1^{t}c_1
=c_2^{t^m}(c_2^{-1})^{t^{m-1}}c_2^{t^{m-1}}(c_2^{-1})^{t^{m-2}}\ldots c_2^{t^2}(c_2^{-1})^tc_2^tc_2^{-1}\\
=c_2^{t^m}c_2^{-1}=[t^m,c_2^{-1}].
\end{aligned},$$ as desired.
 \end{proof}
%
%
%It is now easy to check that  $$c_1= t^{-1} c_1 t$$ and $$c_1c_2 = t^{-1}c_2 t.$$ Consequently, $c_1=c_2^t  c_2^{-1}$. Therefore:
%
%$$\begin{aligned}c_1^m=c_1^{t^{m-1}}c_1^{t^{m-2}}\ldots c_1^{t}c_1
%=c_2^{t^m}(c_2^{-1})^{t^{m-1}}c_2^{t^{m-1}}(c_2^{-1})^{t^{m-2}}\ldots c_2^{t^2}(c_2^{-1})^tc_2^tc_2^{-1}\\
%=c_2^{t^m}c_2^{-1}=[t^m,c_2^{-1}].
%\end{aligned}$$

\begin{lemma}
Let $\Gamma$ be a graph with property (B1), and  $u,v,w\in V(\Gamma)$ with $\lk(v) \subset \st(u) \cap \st(w)$. Consider the $\tau$-map $\tau:=\tau_{u,v,w}$ given by $$\Bigg\{\begin{aligned}
\tau(v)&=v[u,w]\\
\tau(z)&=z,\, z\neq v.\\
\end{aligned}$$
Then $\tau^m\in\SCI A_\Gamma ' (m)$
\label{tau}
\end{lemma}

\begin{proof}
 First, observe that if $u$ and $v$ are connected then $u \in \lk(v) \subset \st(w)$, which implies that the map $\tau$ is the identity. Thus we may assume that $u$ and $v$ are not connected. Note that the hypotheses imply that the transvection $t_{vw}$ is well-defined. Consider also the partial conjugation $c:= c_{u,v}$.
A quick calculation shows that $$\tau=c^{-1}t^{-1}_{vw}ct_{vw}=c^{-1}c^{t_{vw}}.$$ Therefore,  as both $c$ and $c^{t_{vw}}$ lie in
$\SCI A_\Gamma(m)$ we deduce that
  $$\tau^m\SCI A'_\Gamma(m) = (c^{-1} c^{t_{vw}})^m\SCI A'_\Gamma(m)  = c^{-m}(c^m)^{t_{vw}}\SCI A'_\Gamma(m) $$
  which  using Lemma \ref{crossed} gives the desired result.
\end{proof}

\begin{lemma}
Let $\Gamma$ be a graph with property (B2). Let $v, w \in V(\Gamma)$ with  $v \le w$ and consider the transvection $t_{vw}$.
Then, for any integer $m\geq 0$,
$$t_{vw}^{m^2}\in\SCI A'_\Gamma(m)$$
\label{m-transvection}
\end{lemma}

\begin{proof}
By (B2) there exists $u\in V(\Gamma)$ such that $u\neq v,w$ and $v\le u\le w$. There are two cases to consider, depending on whether $u$ and $v$ are connected or not. 
Suppose first that $u$ and $v$ are connected, in which case $u$ and $w$ are also connected since $v\le w$. Therefore, $u$ commutes with both $v$ and $w$ and so
$$[t_{vu}^m,t_{uw}^m]=t_{vw}^{-m^2},$$ which implies $t_{vw}^{m^2}\in\SCI A_\Gamma'(m)$. 

Suppose now that $u$ and $v$ are not connected, and consider  $c:= c_{u,v}$ 
%Let $Y$ be the connected component of $\Gamma - \st(u)$ that contains $v$; as in the previous lemma, we have $Y = \{v\}$. Consider $c:= c_{u,Y}$.
 It is immediate to check that 
$$[t^m_{vu},t^m_{uw}]=c^{-m}(ct_{vw}^{-m})\buildrel m\over\ldots(ct_{vw}^{-m})$$
which, as both $c$ and $t_{vw}^{-m}$ lie in $\SCI A_\Gamma(m)$, yields
$$\SCI A'_\Gamma(m)=[t^m_{vu},t^m_{uw}]\SCI A_\Gamma '(m)=c^{-m}c^mt_{vw}^{-m^2}\SCI A_\Gamma'(m)=t_{vw}^{-m^2}\SCI A'_\Gamma(m),$$ as we wanted to show. 
\end{proof}

We can now prove Proposition \ref{m-torelliab}:

\begin{proof}[Proof of Proposition \ref{m-torelliab}]
Let $m\ge 0$. Consider the decomposition $$A_\Gamma=F_{n_1}\times\ldots\times F_{n_k}\times\Z^a$$ whose existence is guaranteed by Lemma \ref{propertyB}.  As obtained in the proof of that lemma, $n_1,\ldots,n_k$ and $a$ are precisely the cardinalities of the equivalence classes with respect to the relation $\sim$; moreover, we have $a \ne 2$ and $n_i \ne 2$ for all $i$. Now, Lemma \ref{m-torelligen} implies then that $\SCI A_\Gamma(m)$ is normally generated by partial conjugations, $\tau$-maps, and $m$-powers of transvections in $\Aut(A_\Gamma)$. By Lemmas \ref{crossed}, \ref{tau}, and \ref{m-transvection}, we deduce that $m$-th powers of these automorphisms lie in $\SCI A'_\Gamma(m)$. Observe that $\SCI A'_\Gamma(m)$ is normal in $\SCI A_\Gamma$, and  so we have an infinite family of generators all whose $m$-th powers lie in $\SCI A'_\Gamma(m)$. But as $\SCI A_\Gamma(m)/ \SCI A'_\Gamma(m)$ is abelian, this means that the order of any of its elements divides $m$. On the other hand  $\SCI A_\Gamma(m)$ is finitely generated, and thus so is $\SCI A_\Gamma(m)/ \SCI A'_\Gamma(m)$. Hence
 the result follows.
\end{proof}

Armed with the above, we are now in a position to prove Theorem \ref{thma}:

\begin{proof}[Proof of Theorem \ref{thma}] Let $\Gamma$ be a simplicial graph with property (B), and suppose $G\leq \Aut(A_\Gamma)$ is a finite-index subgroup containing the Torelli subgroup $\CI A_\Gamma$, which in turn implies that $\CI A_\Gamma ' \leq G'$. Then there is some $G_1\leq G$, normal in $\Aut(A_\Gamma)$ and which contains $\CI A_\Gamma$, such that the index $[\Aut(A_\Gamma) : G_1]$ is also finite.

Let $m= [\Aut(A_\Gamma) : G_1]$, and observe that for every $\alpha \in \Aut(A_\Gamma)$ we have $\alpha^m \in G_1$. By Lemma \ref{m-torelligen}, $\SCI A_\Gamma(m)$ is normally generated by $\CI A_\Gamma$ and $m$-powers of transvections in $\Aut(A_\Gamma)$. As a consequence, $\SCI A_\Gamma(m)\leq G_1\leq G$ and thus $\SCI A'_\Gamma(m)\leq G'$ also. Since $\SCI A_\Gamma(m)$ and $G$ both have finite index in $\Aut(A_\Gamma)$ and $[\SCI A_\Gamma(m):  \SCI A'_\Gamma(m)]$ is finite by  Proposition \ref{m-torelliab}, we deduce that $[G:G']$ is also finite, which implies the result. \end{proof}

%%%%%%%%%%%%%%%%%%%%%%%%%%%%%

\section{Kazhdan's property T}
\label{s:indic}
In this section we will prove Theorem \ref{prop1}, whose statement we now recall for the reader's convenience.

\begin{named}{Theorem \ref{prop1}} Let $\Gamma$ be a simplicial graph. Then the following conditions are equivalent

\begin{itemize}
\item[i)] $\Gamma$ has property (B2),

\item[ii)] $\Aut(A_\Gamma)/\CI A_\Gamma$  has Kazhdan's property (T),

\item[iii)] $\Aut(A_\Gamma)/\CI A_\Gamma$ is not virtually indicable,

\item[iv)] $H^1(\SAut^0(A_\Gamma)/\CI A_\Gamma,\Z) = 0.$
\end{itemize}
\end{named}
Recall from the introduction that $\SAut^0(A_\Gamma)$ denotes the finite-index subgroup of $\Aut(A_\Gamma)$ generated by transvections and partial conjugations. By a result of Wade (see Corollary 4.10 of \cite{Wade}), the subgroup $\SAut^0(A_\Gamma)$ contains $\CI A_\Gamma$. As we did before, denote by $\mathcal H$ the image of $\SAut^0(A_\Gamma)$ under the  natural homomorphism $\Aut(A_\Gamma) \to \GL(n,\Z)$, where $n$ is the number of vertices of $\Gamma$.

The most involved part in the proof of Theorem \ref{prop1} is to show that the quotient group $\mathcal H$ has Kazhdan's property (T) whenever the graph $\Gamma$ satisfies property (B2). We will do this using a certain inductive argument on the number of equivalence classes of vertices of $\Gamma$ with respect to the equivalence relation $\sim$. To this end, we need to gain a better understanding of the structure of the group $\mathcal H$, and we now proceed to do so.

For the time being, suppose $\Gamma$ is an arbitrary graph with $n$ vertices.  Assume as in the previous section that the classes $\{[v_1],\ldots,[v_k]\}$ are ordered so that $v_j\leq v_i$ implies $j\leq i$. Denote by $n_i$ the cardinality of the class $[v_i]$, and let $V_1,\ldots,V_k$ be the  partition of the set $\{1,\ldots,n\}$ given by  $V_1=\{1,\ldots,n_1\}$, $V_2=\{n_1+1,\ldots,n_1+n_2\}$ and so on.  Consider the directed graph $\Lambda$ with vertices labelled by the $V_i$ and an arrow (i.e. a directed edge) from $V_j$ to $V_i$ whenever $v_j\leq v_i$; in particular, there is an arrow from $V_i$ to itself. We deduce that 
$\mathcal{H}$ is generated by the set $$\{T_{st}\in\GL(n,\Z)\mid s\in V_i,t\in V_j\text{ and there is an arrow }V_j\to V_i\text{ in }\Lambda \}$$
We recall that the elements in $\mathcal{H}$ are block lower triangular matrices and that the diagonal blocks have sizes $n_1,\ldots,n_k$. 

In fact, we may work in the following slightly more general setting. Assume $\{V_1,\ldots,V_k\}$ is a  partition of $\{1,\ldots,n\}$, and let $\Lambda$ be any directed graph so that if there is an arrow $V_j\to V_i$ we have $j\leq i$. We assume also that the graph is transitive, i.e., that if there are arrows $V_j\to V_i$ and $V_i\to V_l$, then there is also an arrow $V_j\to V_l$. We define $\mathcal{H}_\Lambda$ as 
the group generated by the set $$\{T_{st}\in\SL(n,\Z)\mid s\in V_i,t\in V_j\text{ and there is an arrow }V_j\to V_i\text{ in }\Lambda\}.$$
In other words, we work with groups generated by transvections that look like those coming from a right-angled Artin group, without worrying about whether or not this is the case -- doing so will simplify the inductive argument below. 
We stress that all  we said before about the block structure of the matrices in $\mathcal{H}$ remains true for $\mathcal{H}_\Lambda$ (the transitivity assumption is needed for this to be true).

There are two normal subgroups of $\mathcal{H}_\Lambda$ that will be relevant below. Namely,  the subgroup $N_1<\mathcal{H}_\Lambda$, generated by the set 
$$\{T_{st}\in\SL(n,\Z)\mid s\in V_i,t\in V_1\text{ and there is an arrow }V_1\to V_i\text{ in }\Lambda\},$$
and the subgroup $N_2<\mathcal{H}_\Lambda$, generated by the set
$$\{T_{st}\in\SL(n,\Z)\mid s\in V_k,t\in V_j\text{ and there is an arrow }V_j\to V_k\text{ in }\Lambda\}.$$
The block structure of the elements of $\mathcal{H}_\Lambda$ implies that the groups $N_1$ and $N_2$ are also of the type we are considering. Indeed, $N_1=\mathcal{H}_{\Lambda_1}$, where $\Lambda_1$ is the graph obtained from $\Lambda$ by removing all the arrows except those starting in $V_1$; analogously, $N_2=\mathcal{H}_{\Lambda_2}$, with $\Lambda_2$  the graph obtained from $\Lambda$ by removing all the arrows except those ending in $V_k$. 

%
% indeed, if $\Lambda_1$ is the graph obtained from $\Lambda$ by removing all the arrows except those starting in $V_1$, then $N_1=\mathcal{H}_{\Lambda_1}$ and analogousy, 
%$N_2=\mathcal{H}_{\Lambda_2}$  if $\Lambda_2$ is the graph obtained from $\Lambda$ by removing all the arrows except those ending in $V_k$. 

Moreover, using the same reasoning we obtain that 
$$\mathcal{H}_{\Lambda}/N_1\cong\mathcal{H}_{\bar\Lambda_1}\leq\SL(n-n_1,\Z)$$
where $\bar\Lambda_1$ is the graph obtained from $\Lambda$ by removing the vertex $V_1$, and 
$$\mathcal{H}_{\Lambda}/N_2\cong\mathcal{H}_{\bar\Lambda_2}\leq\SL(n-n_k,\Z)$$
where $\bar\Lambda_2$ is the graph obtained from $\Lambda$ by removing the vertex $V_k$.

%(All this follows from the block structure of $\mathcal{H}_\Lambda$).

We now prove the base case for the inductive argument that we will need in the proof of Theorem \ref{prop1}:

\begin{lemma}\label{inductivestep} Assume that one of the following holds:
\begin{itemize}
\item[i)] $\Lambda$ has one arrow $V_1\to V_1$, $n_1>2$ and  there are only arrows starting at $V_1$.

\item[ii)] $\Lambda$ has one arrow $V_k\to V_k$, $n_k>2$ and there are only arrows ending at $V_k$.
\end{itemize}
 Then $\mathcal{H}_\Lambda$ has Kazhdan's property (T).
\label{lemT}
\end{lemma}
\begin{proof} Assume ii) holds. Note that
$$\mathcal{H}_\Lambda=\SL_{n_k}(\Z)\ltimes\mathrm{M}_{n_k\times m}(\Z)$$
where $m=\sum\{n_j\mid j\neq k,\text{ and there is an arrow }V_j\to V_k\}$ and $\SL_{n_k}(\Z)$ acts on $\mathrm{M}_{n_k\times m}(\Z)$ by right matrix multiplication. This group has property (T) by Proposition  1.1 of  \cite{Cornulier}.
The proof for i) is analogous.
\end{proof}

We now furnish the inductive argument that will be used in the proof of Theorem \ref{prop1}:

\begin{proposition}\label{propertyT} Let $\Lambda$ be a graph as above. Assume  that $n_i\neq 2$ for $i=1,\ldots,k$ and that, whenever $n_j=n_i=1$, if  there is an arrow $V_j\to V_i$ with $V_j\neq V_i$  then there is some $r\neq i,j$ and arrows $V_j\to V_r$, $V_r\to V_i$.
Then $\mathcal{H}_\Lambda$  has Kazhdan's property (T).
\end{proposition}
\begin{proof} We claim first that the group $\mathcal{H}_\Lambda$ is perfect, i.e. it has trivial abelianization.  Let $T_{st}$, $s\neq t$, be any element in the generating set used to define $\mathcal{H}_\Lambda$. By definition there are $V_i$ and $V_j$ with $s\in V_i$ and $t\in V_j$ such that there is an arrow $V_j\to V_i$ in $\Lambda$. The claim will follow if we show that there is always some $l\ne s,t$ such that $T_{sl}, T_{lt} \in \mathcal{H}_\Lambda$, for if this were the case we would have
$$T_{st}=[T_{sl},T_{lt}]\in\mathcal{H}_\Lambda'.$$
If $n_i\geq 3$, then we may take $l\in V_i$ with $l\neq s$ and then we have $T_{sl}$,$T_{lt}\in \mathcal{H}_\Lambda$. The same thing happens if $n_j\geq 3$. Thus assume $n_i=n_j=1$. By hypothesis, there is some $r\neq i,j$ and arrows $V_j\to V_r$, $V_r\to V_i$ so we may choose $l\in V_r$. Hence the claim follows. (We remark that this proof also implies that the real version of $\mathcal{H}_\Lambda$ is perfect,
where all the $\SL(n_i,\Z)$-blocks are turned into $\SL(n_i,\R)$ blocks and all the $\mathrm{M}_{n_i\times n_j}(\Z)$-blocks are turned into  $\mathrm{M}_{n_i\times n_j}(\R)$-blocks.)

Next, we now proceed by induction on $n$. The argument boils down to removing either the initial or terminal vertex of $\Lambda$, thus passing to groups of the same type that are contained in $\SL(m,\Z)$ for $m<n$. 
%and the way to it will be removing either the initial or the final vertex of $\Lambda$. 
Observe that   the  graph obtained by removing such a vertex still satisfies the same properties on the arrows. 

First, observe that the result is trivial if $\Gamma$ has only one vertex, so assume this is not the case. 

Let, with the same notation used above,
$$N_2=\mathcal{H}_{\Lambda_2}$$
where $\Lambda_2$ is the graph obtained from $\Lambda$ by removing all the arrows except of those ending in $V_k$. 
Assume first that  $n_k\geq 3$, in which case  Lemma \ref{lemT} implies that $N_2$ has property (T). On the other hand, 
$$\mathcal{H}_{\Lambda}/N_2\cong\mathcal{H}_{\bar\Lambda_2}\leq\SL(n-n_k,\Z)$$
where $\bar\Lambda_2$ is the graph obtained from $\Lambda$ by removing the vertex $V_k$. Observe that the hypotheses on $\Lambda$ also holds for $\bar\Lambda_2$ so by induction we may assume that  $\mathcal{H}_\Lambda/N_2$  has property (T). At this point we deduce from  Proposition 1.7.6  of \cite{BHV} that $\mathcal{H}_\Lambda$ has property (T) also. 

Proceeding in an analogous manner with $N_1$ instead of $N_2$ gives the desired result whenever $n_1 \ge 3$. Therefore we may assume that $n_1=n_k=1$. 
At this point, let $C$ be either the trivial group, in the case when there is no arrow $V_1\to V_k$, or the subgroup generated by $T_{n_1}$ otherwise. Observe that in both cases $C$ is central in $\mathcal{H}_\Gamma$ (since the upper and the lower diagonal block in any matrix in $\mathcal{H}_\Lambda$ are each a single 1).  Moreover, $C=N_1\cap N_2$. 

%Using again the block structure of the matrices in $\mathcal{H}_\Lambda$ we see that
%$$\mathcal{H}_\Lambda/N_1N_2\cong \mathcal{H}_{\bar\Lambda}\leq\SL(n-2,\Z)$$
%where $\bar\Lambda$ is the graph obtained from $\Lambda$ by removing both vertices $V_1$ and $V_k$. (Observe that $\bar\Lambda\neq\emptyset$.) By induction, $\mathcal{H}_\Lambda/N_1N_2$ has property (T). 

As before, $\mathcal{H}_\Lambda/N_1$ and $\mathcal{H}_\Lambda/N_2$  have property (T). 
%If, say, $\mathcal{H}_\Lambda/N_1$ was trivial this would imply that either $\Lambda$ has only two vertices $V_1$ and $V_k=V_2$ or that there are %more vertices $V_2,\ldots,V_{k-1}$, all of them with $n_i=1$ and but there are no arrows between the vertices $V_2,\ldots,V_k$. In both cases if there %was some arrow from $V_1$ to any other $V_i$, $i\neq 1$ this would contradict the hypothesis. Therefore there is no arrow at all in $\Lambda$ thus 
%$\mathcal{H}_\Lambda=1$. 
%So we assume that $\mathcal{H}_\Lambda/N_1$ and $\mathcal{H}_\Lambda/N_2$ are not trivial thus have (T).
Observe that
$$\mathcal{H}_\Lambda/C\cong\mathcal{H}_\Lambda/N_1N_2\ltimes(N_1/C\oplus N_2/C)$$
where we regard $\mathcal{H}_\Lambda/N_1N_2$ as matrices in $\SL(n-2,\Z)$ acting via left multiplication on $N_1/C$ and via right multiplication on $N_2/C$. 
Now, as $\mathcal{H}_\Lambda/N_1$ has property (T), we deduce from Remark 1.7.7 in \cite{BHV} that both $\mathcal{H}_\Lambda/N_1N_2$ and the pair $(\mathcal{H}_\Lambda/N_1,N_1N_2/N_1)$ have property (T). Similarly, since  $\mathcal{H}_\Lambda/N_2$ has property (T), we deduce from the aforementioned remark in \cite{BHV} that the same is true for the pair $(\mathcal{H}_\Lambda/N_2,N_1N_2/N_2)$. Moreover, we have
$$\mathcal{H}_\Lambda/N_1=\mathcal{H}_\Lambda/N_1N_2\ltimes N_2/C,\text{ and}$$
$$\mathcal{H}_\Lambda/N_2=\mathcal{H}_\Lambda/N_1N_2\ltimes N_1/C.$$
At this point Lemma 5.2 of \cite{Fernos} implies that the pair  $(\mathcal{H}_\Lambda/C,N_1/C\oplus N_2C)$ also has property (T), and Remark 1.7.7 in \cite{BHV} implies that the same is true for  $\mathcal{H}_\Lambda/C$. In the case when $C=1$ this finishes the proof; otherwise, Theorem 1.7.11 of \cite{BHV} and  the first part of the proof yield the desired result.
\end{proof}

\begin{proof}[Proof of Theorem \ref{prop1}] We first prove that (i) implies (iv). Seeking a contradiction, assume that iv) holds but that $\Gamma$ does not have property (B2); in other words, there are vertices
$v\neq w$  with $v\leq w$ such that there is no vertex $u\ne v$ with $v\leq u\leq w$. In particular, $[v]=\{v\}$ and $[w]=\{w\}$, where $[\cdot]$ denotes the equivalence class with respect to the equivalence relation $\sim$. Now, it follows from Wade's presentation of the group $\mathcal H$ (see Proposition 4.11 of \cite{Wade}), that whenever the image $T_{vw}$ of the transvection $t_{vw}$   in $\mathcal H$  appears in some relator of $\mathcal H$, it does so with exponent sum 0.  Therefore the homomorphism
$$\pi:\mathcal{H}\to\Z$$
given by $\pi(T_{vw})=1$ (with additive notation in $\Z$) and $\pi(T_{rs})=0$ for any other generator $T_{rs}$ is well-defined. In particular, $H^1(\SAut^0(A_\Gamma)/\CI A_\Gamma,\Z)\neq 0$ which contradicts iv). 

Obviously ii) implies iii) and iii) implies iv) so it suffices to show that (ii) implies (iii), i.e. that if $\Gamma$ has property (B2) then 
$\Aut(A_\Gamma)/\CI A_\Gamma$ has property (T). As  property (T) is invariant under finite index extensions, (see Theorem 1.7.1 of \cite{BHV}), it is enough to show that $\mathcal{H}$ has property (T).
As $\Gamma$ has property (B2),  Lemma \ref{propertyB} implies that there is no class $[v]$ with exactly two elements. Therefore, if $\Gamma$ has (B2) then the associated graph $\Lambda$ satisfies the hypotheses of Proposition \ref{propertyT}, and hence $\mathcal{H}=\mathcal{H}_\Lambda$ has property (T).
\end{proof}

\section{Virtual indicability}

We now proceed to give a proof of Theorem \ref{thm-indic2}. Recall from the introduction that  $\Aut^1 A_\Gamma$ is the finite-index subgroup of $\Aut(A_\Gamma)$  generated by   transvections, partial conjugations and thin inversions. Day's presentation of $\Aut(A_\Gamma)$, described in Theorem \ref{thm-day} above, implies that $\Aut^1( A_\Gamma)$ contains precisely those graphic automorphisms that preserve each equivalence class with respect to the equivalence relation $\sim$ and fix thin vertices.  We denote  the subgroup of such graphic automorphisms by $\operatorname{Sym}^1(A_\Gamma)$.

In order to prove Theorem \ref{thm-indic2} we will need to work with a presentation of  $\Aut^1(A_\Gamma)$. To this end, a modification of Day's arguments in \cite{Day} implies the following: 

\begin{proposition}[\cite{Day}] The group $\Aut^1(A_\Gamma)$ has a finite presentation with generators the set of type (2) Whitehead automorphisms and $\operatorname{Sym}^1(A_\Gamma)$, and relators (R1), (R2), (R3), (R4), (R5), (R9), (R10) above together with
\begin{enumerate}
\item[(R6)'] $\sigma (A,a) \sigma^{-1}= (\sigma(A), \sigma(a))$, for every $\sigma\in\operatorname{Sym}^1(A_\Gamma)$.
\item[(R7)'] All the relations among automorphisms in $\operatorname{Sym}^1(A_\Gamma)$.
\end{enumerate}
\label{prop:day}
\end{proposition}

%\begin{proposition}[\cite{Day}] 
%$\Aut^1( A_\Gamma)$ is finitely presented. Moreover, a finite generating set consists of all the elements of $\operatorname{Sym}^1(A_\Gamma)$, all the type (2) Whitehead automorphisms, and all inversions. A complete set of relations is given by the ser $R^0$ consisting of those relations of types (R1) -- (R10) described in Theorem \ref{thm-day}, subject to the requirement that the only type (1) Whitehead automorphisms that appear in a relator must be elements of  $\operatorname{Sym}^1(A_\Gamma)$.
%\label{prop:day}
%\end{proposition}

\begin{proof}[Sketch proof of Proposition \ref{prop:day}]
First, it follows directly from  the definition that $\Aut^1(\Gamma)$ is generated by all the type (2) Whitehead automorphisms, and every thin inversion. Thanks to relator (R5), we may add the elements of $\operatorname{Sym}^1(A_\Gamma)$ to this list of generators.

Let $R^1$ be the list (R1)-(R10) of Day's relators, except that (R6) and (R7) are substituted by (R6)' and (R7)', respectively. Observe that every relation in $R^1$ is indeed a relation in $\Aut^1(A_\Gamma)$. Therefore, it remains to justify why these form a complete set of relations in $\Aut^1(A_\Gamma)$. 

By Theorem A.1 of \cite{Day}, every automorphism $\alpha \in \Aut(A_\Gamma)$ may be written as a product $\alpha= \beta \gamma$, where $\beta$ lies in the subgroup of  $\Aut(A_\Gamma)$ generated by {\em short-range} automorphisms and $\gamma$ is in the subgroup generated by {\em long-range} automorphisms.  Here, we say that $\delta \in \Aut(A_\Gamma)$ is {\em long-range} if either it is a type (1) Whitehead automorphism, or it is a type (2) automorphism  specified by a subset $(A,v)$ such that $\delta$ fixes all the elements adjacent to $v$ in $\Gamma$. Similarly, we say that  $\delta \in A_\Gamma$ is {\em short-range} if it is a type (2) automorphism specified by a subset $(A, v)$ and $\delta$ fixes all the elements of $\Gamma$ not adjacent to $v$. Following Day, we denote by $\Omega_l$ (resp. $\Omega_s$) the set of all long-range (resp. short-range) automorphisms.

Consider now $\alpha \in \Aut^1(A_\Gamma)$, and observe that all short-range automorphisms are in $\Aut^1(A_\Gamma)$. The proof of the splitting in Theorem A.1 of \cite{Day} above is based in the so called {\em sorting substitutions} in \cite{Day} Definition 3.2. Of these, only substitution (3.1) involves an element possibly not in $\Aut^1(A_\Gamma)$ and this element is just moved along, meaning that if our initial string consists  solely of elements in $\Aut^1(A_\Gamma)$, then so does the final  string. Moreover, observe that the relators needed for these moves all lie in $R^1$ (an explicit list of the relators needed, case by case, can be found in \cite[Lemma 3.4]{Day}). All this implies that up to conjugates of relators in $R^1$, we may write $\alpha= \beta \gamma$ with  $\beta$ in the subgroup of  $\Aut^1(A_\Gamma)$ generated by $\Omega_s$ and $\gamma$  in the subgroup generated by  $\Omega_l^1=\Omega_l \cap \Aut^1(A_\Gamma)$. 

By Proposition 5.5 of \cite{Day}, the subgroup of $\Aut^1(A_\Gamma)$ generated by $\Omega_s$ has a presentation whose every generator is a short-range automorphism or an element of $\operatorname{Sym}^1(A_\Gamma)$, and whose every relator lies in $R^1$. Indeed, in the proof of  \cite[Proposition 5.5]{Day}, the generators that we need to add to $\Omega_s$ to get the desired presentation are precisely the elements of the form $\sigma_{ab}$ of (R5), which belong to $\operatorname{Sym}^1(A_\Gamma)$.

In addition, the subgroup $\Aut^1(A_\Gamma)$ generated by $\Omega^1_l $ has a presentation whose every relator is in $R^1$. To see that this is indeed the case, first recall from Proposition 5.4 of \cite{Day} that the subgroup of $\Aut(A_\Gamma)$ generated by $\Omega_l$ admits a presentation in which every relation (also in the list (R1)-(R10) of Theorem \ref{thm-day}) is written in terms of $\Omega_l$. In order to prove this, Day uses a certain inductive argument called the {\em peak reduction algorithm}. But by Remark 3.22 of \cite{Day}, every element of $\Aut^1(A_\Gamma)$ may be peak-reduced using elements of $\Aut^1(A_\Gamma)$ {\em only}. Indeed, the only subcase of  \cite[Remark 3.22]{Day} which is problematic in this setting is the use of subcase (3c) of \cite[Lemma 1.18]{Day}. But the relator used in that subcase is precisely (R5) where the type (1) Whitehead automorphism is $\sigma_{ab}$, which belongs to $\operatorname{Sym}^1(A_\Gamma)$.

Moreover, the process of peak reduction needs relators in $R^1$ only; this is a consequence of the fact, observed already in Remark 3.22 of \cite{Day}, that  type (1) Whitehead automorphisms are only moved around when lowering peaks, and if they lie in $\Omega_l^1$ then the needed relator is precisely (R5), where the type (1) Whitehead automorphism is $\sigma_{ab}\in\operatorname{Sym}^1(A_\Gamma)$.
\end{proof}

\begin{remark} Essentially the same proof above can be used to show that the group $\Aut^0(A_\Gamma)$ generated by transvections, partial conjugations and all the inversions admits an analogous finite presentation as the previous one; however, here one needs to include all those graphic automorphisms $\operatorname{Sym}^0(A_\Gamma)$ that preserve the equivalence classes for $\sim$.
\end{remark}

We are now in a position to prove Theorem \ref{thm-indic2}, whose statement we now recall for the reader's convenience:

\begin{named}{Theorem \ref{thm-indic2}}
Let $\Gamma$ be a simplicial graph. Suppose there exists a vertex $w\in V(\Gamma)$ such that
there is no $v \in V(\Gamma)$ with $v\le w$. Then $\Aut^1(A_\Gamma)$  surjects onto $\Z$. If moreover
$\Gamma - \st(w)$ is not connected, then also $\Out^1(A_\Gamma)$ surjects onto $\Z$.
\end{named}

\begin{proof}[Proof of Theorem \ref{thm-indic2}] We are going to construct an explicit surjective homomorphism $\Aut^1(A_\Gamma) \to \Z$. 
Before proceeding to do so, observe that the fact that there is no $v \in V(\Gamma)$ with $v\le w$ implies $[w]=\{w\}$  and  $\Gamma\neq\st(w)$. In particular, $w$ is thin. Let $Y$ be a connected component of $ \Gamma-\st(w)$ and consider the partial conjugation $c_{w,Y}$; as mentioned in the paragraph before Theorem \ref{thm-day}, in terms of Whitehead automorphisms we write $$c_{w,Y}=(Y\cup Y^{-1}\cup\{w\},w).$$ 
We claim that for any pair $v_1,v_2\not\in\st(w)$ such that $v_1\in Y$ and $[v_1]= [v_2]$ we must have $v_2\in Y$. Indeed, since $v_1\not\leq w$, there exists some $z\in\lk(v_1)$ with $z\not\in\st(w)$. But $v_1\sim v_2$, and hence $z\in\st(v_2)$. Therefore  $v_1$ and $v_2$ are connected in $\Gamma-\st(w)$ and so  $v_2\in Y$, as desired. 

In the light of the above claim, any connected component of $\Gamma - \st(w)$ is a union of sets of the form $[v]\cap(\Gamma-\st(w))$. Consider any graphic automorphism $\sigma\in\operatorname{Sym}^1( A_\Gamma)$. As $[w]=\{w\}$ we must have $\sigma(w)=w$. It also follows from the definition of $\operatorname{Sym}^1(A_\Gamma)$  that, for any vertex $v$ and any $\sigma \in \operatorname{Sym}^1(A_\Gamma)$, $\sigma$ preserves each $[v]\cap(\Gamma-\st(w))$ setwise, and hence also every connected component of $\Gamma - \st(w)$.

Let $\pi_Y$ be the map which from the set of Whitehead automorphisms to $\Z$ which is defined as follows: given a  Whitehead automorphism $g\in \Aut^1(A_\Gamma)$, we set
$$\pi_Y(g)=\left \{\begin{aligned}
&1\text{ if }g=(A,w)\text{ with }Y\cup Y^{-1}\subseteq A\\
-&1\text{ if }g=(A,w^{-1})\text{ and }Y\cup Y^{-1}\subseteq A\\
&0\text{ otherwise.}  \\
\end{aligned}
\right.$$
We claim that $\pi_Y$ yields a well defined surjective homomorphism $\Aut^1(A_\Gamma\to\Z)$. In order to prove the claim,  we go through the list of Day's relations (R1) -- (R9) described in Theorem \ref{thm-day}, and  check that each of them is preserved by $\pi_Y$. Observe it is immediate that $\pi_Y$ preserves (R1), (R2), (R3), (R7)' and (R9). Relation (R5) is  trivial since there is no $v\ne w$ with $v\sim w$, from the hypotheses on $\Gamma$, and therefore there is nothing to verify. Finally, we check that $\pi_Y$ preserves (R4); to this end, the only problematic case is \[ (B, v)(A,w)(B,v)^{-1} =  (A,w)((B - v) \cup w, w),\] whenever $\{w,w^{-1}\} \cap B = \emptyset$, $v \notin A$, $v^{-1} \in A$,  and at least one of $A \cap B = \emptyset$ or $ v \in \lk(w)$ holds. Observe that, unless $Y\cup Y^{-1} \subset B$, this relator is trivially preserved by $\pi_Y$, so we assume that $Y\cup Y^{-1} \subset B$. At this point, the fact that $(B,v)$ and $((B - v) \cup w, w)$ are both defined implies that $v\le w$; see, for instance, Lemma 2.5 of \cite{Day} for a proof. By the hypotheses on $\Gamma$, we deduce that $v=w$, which contradicts the conditions on $A$ and $B$ of relation (R4). 
%
%Relator (R6)' is the reason why we have to argue with $\Aut^1(A_\Gamma)$ instead of $\Aut^0(A_\Gamma)$. Essentially, since all the elements in $\operatorname{Sym}^1(A_\Gamma)$ fix $w$ (because $w$ is thin), (R6)' is obviously preserved by $\pi_Y$. 

Next, (R6)' is obviously preserved by $\pi_Y$ since all the elements in $\operatorname{Sym}^1(A_\Gamma)$ fix $w$, because $w$ is thin.

Now, consider (R10). The only case when this relator gives problems is when there is some $(A,w)$ or $(A,w^{-1})$ well defined and a $v\in L$ with $v\in A$, $v^{-1}\not\in A$ because the we would have 
$$(A,w)(L-v^{-1},v)(A,w)^{-1}=(L-v^{-1},v)(L-w^{-1},w)$$
or the same thing but with $(A,w^{-1})$ instead of $(A,w)$. But then the condition $v\in A$, $v^{-1}\not\in A$ implies, again by Lemma 2.5 of \cite{Day}, that $v\leq w$ which is impossible unless $v=w^{\pm 1}$. But this case has to be excluded from (R10) since in other case it would mean that all the conjugations $(L-v^{-1},v)$ vanish (in fact in that case we would be in $\Out(A_\Gamma)$).

Now, assume that $\Gamma - \st(w)$ is not connected. Then we may find a  connected component $Z\ne Y$. We may repeat the construction above, obtaining a surjective homomorphism $\pi_Z: \Aut^1(A_\Gamma) \to \Z$. 
We claim that  map $$\pi:= \pi_Y - \pi_Z :H\to\Z$$ can be lifted to a well-defined surjective homomorphism $\Out^1(A_\Gamma) \to \Z$; to this end, the only  thing to verify is that $\pi$ annihilates all the conjugations $(L-v^{-1},v)$. The definition of $\pi_L$ and $\pi_K$ implies that we need to check only that 
$\pi(L-w^{-1},w)$ and $\pi(L-w,w^{-1})$ vanish. But as $Y\cup Y^{-1}$ and $Z\cup Z^{-1}$ are contained in  $L-w^{-1}$, we have
\begin{equation} \pi(L-w^{-1},w)=\pi_Y(L-w^{-1},w)-\pi_Z(L-w^{-1},w)=0,
\label{eq:tau1}
\end{equation} and
\begin{equation} \pi(L-w,w^{-1})=\pi_Y(L-w,w^{-1})-\pi_Z(L-w,w^{-1})=0,
\label{eq:tau2}
\end{equation} as desired. This finishes the proof of Theorem \ref{thm-indic2}.
\end{proof}

\begin{remark}
The reason why the proof above breaks if one considers $\Aut^0(A_\Gamma)$ instead of $\Aut^1(A_\Gamma)$ is precisely relator  (R6'). Indeed, we are using the fact that all the elements in $\operatorname{Sym}^1(A_\Gamma)$ fix $w$, because $w$ is thin; however, this need not be true if we take $\operatorname{Sym}^0(A_\Gamma)$ instead. 
\end{remark}

%Relator (R6)' is the reason why we have to argue with $\Aut^1(A_\Gamma)$ instead of $\Aut^0(A_\Gamma)$. Essentially, since all the elements in $\operatorname{Sym}^1(A_\Gamma)$ fix $w$ (because $w$ is thin), (R6)' is obviously preserved by $\pi_Y$. 

\iffalse
Finally observe that equations (\ref{eq:tau1}) and (\ref{eq:tau2}) imply that the subgroup $\Inn(A_\Gamma)$ of inner automorphisms of $A_\Gamma$ is contained in the kernel of $\pi$. In particular, as mentioned in the introduction, we obtain that of $\Out(A_\Gamma)$ is virtually indicable: 

\begin{corollary}
Let $\Gamma$ be a simplicial graph. Suppose there exists $w\in V(\Gamma)$, with $\Gamma-\st(w)$ disconnected, and such that there is no $v \in V(\Gamma)$ with $v\le w$. Then  $${\rm Out}^0(A_\Gamma):= \Aut^0(A_\Gamma) / \Inn(A_\Gamma)$$  surjects onto $\Z$.
\label{cor:out}
\end{corollary}
\fi

Finally, as we also mentioned in the introduction, it is easy to give a characterization of when the hypotheses of Theorem 
 \ref{thm-indic2} are satisfied in the  case when the graph $\Gamma$ is a tree. Recall that a {\em leaf} of a tree is a vertex whose link has exactly one element. We have:

 \begin{proposition} Let $\Gamma$ be a connected tree. Then $\Gamma$ the following conditions related to the hypotheses of Theorem  \ref{thm-indic2} are equivalent:
 \begin{itemize}
\item[i)] there is a minimal vertex $w$,

\item[ii)] there is a minimal vertex $w$ such that  $\Gamma-\st(w)$ is disconnected,

\item[iii)] there is a vertex $w$ whose distance to each leaf of $\Gamma$ is at least 3.
\end{itemize}
 \label{tree-charac}
\end{proposition}
\begin{proof} Observe that for any vertex $w$ there is some $v\neq w$ with $v\leq w$ if and only if $d(v,w) \in \{1,2\}$  and $v$ is a leaf.  Therefore the existence of a minimal vertex $w$ is equivalent to the existence of a vertex at distance at least 3 to each leaf. Moreover, if the distance from a vertex $w$ to each leaf is at least 3, then $\Gamma-\st(w)$ is disconnected.
\end{proof}

\section{Appendix. On the linearity problem for $\Aut(A_\Gamma)$}
\label{appendix}

A question of Charney (see Problem 14 of \cite{Charney})  asks for which graphs $\Gamma$ is $\Aut(A_\Gamma)$ linear; recall that a group $H$ is said to be linear if there is an injective homomorphism $H \to \GL(n,K)$ for some field $K$ and some $n>0$. 

During our work, we noticed that if the graph $\Gamma$ satisfies a certain (drastic) weakening of property (B), then an argument of Formanek-Procesi \cite{FP} immediately yields that $\Aut(A_\Gamma)$ is not linear, thus offering a refinement of Charney's question. 

We give a short account of Formanek-Procesi's result for the sake of completeness, following the summary given in \cite{BH}. Let $H$ be a group, and let $v_1,v_2, v_3\in H$ such that $\langle v_1,v_2,v_3 \rangle \cong F_3$, the free group on three generators. Let $\alpha_1,\alpha_2 \in \Aut(H)$ be such that 

$$\Bigg\{\begin{aligned}
\alpha_i(v_j)&=v_j, {\text{ for }} i,j=1,2\\
\alpha_i(v_3)&=v_3v_i,  {\text{ for }} i=1,2.\\
\end{aligned}.$$

The group $\langle \alpha_1, \alpha_2\rangle$ is called a {\em poison subgroup} of $\Aut(H)$; observe that  $\langle \alpha_1, \alpha_2\rangle\cong F_2$. The result of Formanek-Procesi \cite{FP} asserts: 

\begin{theorem}[\cite{FP}]
Let $H$ be a group. If $\Aut(H)$ contains a poison subgroup, then $\Aut(H)$ is not linear. 
\end{theorem} 

In \cite{FP}, Formanek-Procesi used the result above to prove that the (outer) automorphism group of the free group $F_n$ on $n\ge 3$ generators contains poison subgroups, and thus is not linear. 
We now mimic their reasoning, and introduce a property of a simplicial graph $\Gamma$ that guarantees that  $\Aut(A_\Gamma)$ has a poison subgroup, and thus is not linear either. 

\begin{definition}
Let $\Gamma$ be a simplicial graph. We say that $\Gamma$ satisfies property (NL) if there exist pairwise non-adjacent vertices $v_1, v_2, v_3 \in V(\Gamma)$ such that 
$v_3 \le v_i$, for $i=1,2$. 
\label{defN}
\end{definition}

We now prove Proposition \ref{nonlinear}, whose statement we now recall: 

\begin{named}{Proposition \ref{nonlinear}}
Let $\Gamma$ be a simplicial graph that satisfies property (NL). Then $\Aut(A_\Gamma)$ contains a poison subgroup and thus is not linear. 
\end{named}

\begin{proof}
Let $\Gamma$ be a simplicial graph with property (NL), and let $v_1,v_2,v_3\in V(\Gamma)$ be three pairwise non-adjacent vertices; by definition, $$\langle v_1,v_2,v_3 \rangle \cong F_3.$$ Since $v_3 \le v_i$ for $i=1,2$, we may consider now the transvections $t_{v_3v_i}$ 
for $i=1,2$. It  follows that the group generated by $\alpha_1:= t_{v_3v_1}$ and $\alpha_2:=t_{v_3v_2}$ is a poison subgroup of $\Aut(A_\Gamma)$, as desired. 
\end{proof}

%It is worth mentioning that, using the exact same argument, we obtain that $\Out(A_\Gamma)$ is not linear whenever $\Gamma%$ satisfies property (NL).

\begin{remark}
Property (NL) is more frequent than one could in principle have thought. For instance, it is satisfied as soon as $\Gamma$ has three isolated vertices, in which case $\Aut(F_3) < \Aut(A_\Gamma)$. We remark, however, that an arbitrary graph with property (NL) need not have
 any isolated vertices, see Figure \ref{linear} for an example.

\begin{figure}[htb]
\begin{center}

\includegraphics[width=3.3in,height=1.2in]{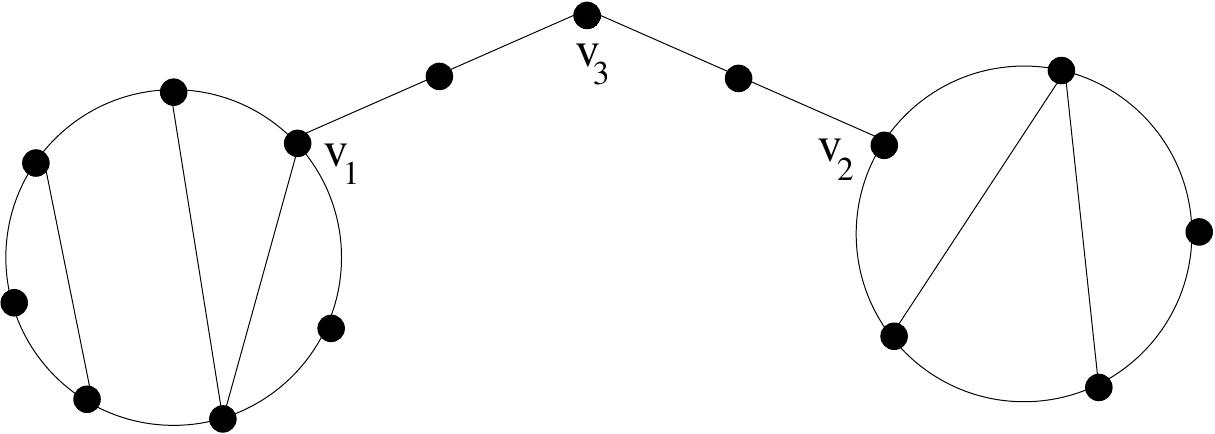} \caption{A graph $\Gamma$ with property (NL) and no isolated vertices } \label{linear}
\end{center}
\end{figure}

Following Erd\H{o}s-R\'enyi \cite{ER}, we denote by $G(n,N)$ the space of graphs with $n$ vertices and $N$ edges, where we choose graphs $\Gamma_{n,N}\in G(n,N)$ at random with respect to the uniform probability distribution in $G(n,N)$. Fix a constant $c \in \R$, and set $N(n) = \frac{1}{2} n \log(n) + c n$. Erd\H{o}s-R\'enyi \cite{ER} showed that the probability $P_{n, N(n)}(k)$ that a random graph $\Gamma_{n,N(n)}$ has $k$ isolated vertices behaves, as $n$ grows, as a Poisson distribution with parameter $\lambda= e^{-2c}$; see Theorem 2c of \cite{ER}. More concretely one has: 

 $$\displaystyle{\lim_{n\to \infty}} P_{n, N(n)}(k) = \frac{\lambda^ke^{-\lambda}}{k!}.$$ 

In the light of this result, the probability that a random graph $\Gamma_{n,N(n)}$ satisfies property (NL) is  strictly bigger than $$\frac{\lambda^3e^{-\lambda}}{3!}, $$ where again $\lambda= e^{-2c}$. Thus we see that, at least for that particular $N(n)$, a definite (albeit small) proportion of random graphs satisfy property (NL).
\end{remark}

\end{document}